\documentclass[12pt]{article}
\usepackage{amssymb, amsmath, amsthm, amsfonts, tikz, geometry, ytableau, mathrsfs}
\usepackage[colorlinks,linkcolor=blue,anchorcolor=blue,citecolor=blue]{hyperref}

\numberwithin{equation}{section}
\numberwithin{figure}{section}

\hypersetup{colorlinks = true}
\newtheorem{thm}{Theorem}[section]

\newtheorem{lem}[thm]{Lemma}

\newtheorem{exam}[thm]{Example}
\newtheorem{rem}[thm]{Remark}
\allowdisplaybreaks

\begin{document}
\begin{center}
	{\large \bf Planar networks and total positivity of Riordan arrays}
\end{center}

\begin{center}
Lanqi Du$^{1}$, Ethan Y.H. Li$^{2}$\\[6pt]
\end{center}

\begin{center}
School of Mathematics and Statistics,\\
Shaanxi Normal University, Xi'an, Shaanxi 710119, P. R. China \\[6pt]
Email: $^{1}${\tt dulanqi@snnu.edu.cn} $^{2}${\tt yinhao\_li@snnu.edu.cn}
\end{center}

\noindent\textbf{Abstract.}
In 2015, Chen, Liang and Wang provided several sufficient conditions for the total positivity of Riordan arrays and asked for combinatorial proofs of these results. In this paper, we present such proofs by constructing suitable planar networks with non-negative weights and applying the Lindstr\"om-Gessel-Viennot lemma. Moreover, we slightly generalize one of the results and give more totally positive Riordan arrays.
\vspace{0.5cm}



\section{Introduction}\label{sec-intro}

As elements of the Riordan group, the Riordan arrays unify a lot of themes in enumerative combinatorics \cite{SGWW91}. Let $g(x) = \sum_{n \ge 0}g_nx^n$ and $f(x) = \sum_{n\ge 0}f_nx^n$ be two formal power series with $g_0\neq 0$, $f_0 = 0$ and $f_1 \neq 0$. Then a (proper) \textit{Riordan array} $(g,f)$ is defined as the matrix $R=(r_{n,k})_{n,k \ge 0}$ such that the generating function of the $k$th column is $g(x)f^k(x)$, i.e.,
\[
r_{n,k} = [x^n] g(x)f^k(x).
\]
As an alternative description, a Riordan array can be characterized by the $A$-sequence $(a_n)_{n \ge 0}$ and $Z$-sequence $(z_n)_{n \ge 0}$ \cite{MRSV97,Sp94}, which satisfy the following recurrence relations:
\[
r_{0,0} = 1, \quad r_{n+1,0} = \sum_{i \ge 0} z_i r_{n,i}, \quad r_{n+1,k+1} = \sum_{i \ge 0} a_i r_{n,k+i},
\]
see \cite{Bar25,He15,HS09} for more information.

Among the various properties of matrices, total positivity has received significant research interests. This concept plays an important role in many fields, such as convexity, approximation theory, real analysis, combinatorics and Lie group theory, see \cite{GK50,Kar68,Pin10} and the references therein. Following Karlin \cite{Kar68}, we call a matrix \textit{totally positive of order $r$} ($\textup{TP}_r$) if all its minors of order at most $r$ are non-negative, and call it \textit{totally positive} ($\textup{TP}$) if all its minors are non-negative. In particular, the total positivity of Riordan arrays has also been extensively studied, see \cite{CLW15,CS24,CW19,HS24,HS25,MMW22,Slo20,Zhu21} for instance. Especially, Chen, Liang and Wang \cite{CLW15} transformed this issue into the study of production matrices.
\begin{thm}[{\cite[Theorem 2.1]{CLW15}}]\label{thm-prod-mat-tp}
  Let $R$ be a Riordan array whose $A$- and $Z$-sequences are $(a_n)_{n \ge 0}$ and $(z_n)_{n \ge 0}$ respectively.
\begin{itemize}
\item[(i)] If its production matrix
\[
    J(R) =
    \begin{pmatrix}
    z_0 & a_0 & \quad & \quad & \quad\\
    z_1 & a_1 & a_0 & \quad & \quad\\
    z_2 & a_2 & a_1 & a_0 & \quad\\
    \vdots & \vdots & \quad & \quad & \ddots
    \end{pmatrix}
\]
is $\textup{TP}$ (resp. $\textup{TP}_r$), then $R$ is also $\textup{TP}$ (resp. $\textup{TP}_r$).
\item[(ii)] If $J(R)$ is $\textup{TP}_2$ and all $z_n \ge 0$, then $(r_{n,0})_{n \ge 0}$ is log-convex, i.e., $r_{i,0}^2 \le r_{i-1,0}r_{i+1,0}$ for $i \ge 1$.
\end{itemize}
\end{thm}
Using Theorem \ref{thm-prod-mat-tp}, they also proved the following criteria in the special case that $J(R)$ is a tri-diagonal matrix (also called a Jacobi matrix).
\begin{thm}[{\cite[Theorem 2.3 and Proposition 2.6]{CLW15}}]\label{thm-abst}
  Let $a,b,r,s,t$ be five non-negative numbers and $R(a,b;r,s,t)$ be the Riordan array with $A = (r,s,t,0,0,\ldots)$ and $Z = (a,b,0,0,\ldots)$.
  \begin{itemize}
    \item [(i)] If $as \ge br$ and $s^2 \ge rt$, then $R(a,b;r,s,t)$ is $\textup{TP}_2$ and $(r_{n,0})_{n \ge 0}$ is log-convex.
    \item [(ii)] If $s^2 \ge 4rt$ and $a(s+\sqrt{s^2-4rt})/2 \ge br$, then $R(a,b;r,s,t)$ is totally positive.
  \end{itemize}
\end{thm}
Moreover, Chen, Liang and Wang \cite{CLW15} also considered consistent and quasi-consistent Riordan arrays. Precisely, a Riordan array $R$ with $A$-sequence $(a_n)_{n\ge 0}$ is called \textit{consistent} (resp. \textit{quasi-consistent}) if its $Z$-sequence is $(a_n)_{n\ge 0}$ (resp. $Z = (a_n)_{n \ge 1}$). In both cases the total positivity of the corresponding production matrices relies only on the property of the sequence $A=(a_n)_{n\ge0}$. Specifically, if the Toeplitz matrix of $A$
\[
    (a_{i-j})_{i,j \ge 0} =
    \begin{pmatrix}
    a_0 & \quad & \quad & \quad\\
    a_1 & a_0 & \quad & \quad\\
    a_2 & a_1 & a_0 & \quad\\
    \vdots & \quad & \quad & \ddots
    \end{pmatrix}
\]
is $\textup{TP}_r$, then $A$ is said to be a \textit{P\'olya frequency sequence of order $r$} ($\textup{PF}_r$). Similarly, if $(a_{i-j})_{i,j \ge 0}$ is $\textup{TP}$, then $A$ is said to be a \textit{P\'olya frequency sequence} ($\textup{PF}$). These concepts are intensively studied by Schoenberg, Edrei and their collaborators, see \cite{ASW52,Bre89,Kar68,WY05} for surveys and more results. In particular, if $(a_{i-j})_{i,j \ge 0}$ is $\textup{TP}_2$ then $A$ is a log-concave sequence, which is an important concept in combinatorics and closely related to unimodality and real-rootedness, see \cite{Bre89,Sta89}.

\begin{thm}[{\cite[Theorem 2.7]{CLW15}}]\label{thm-PF}
Let $R$ be a consistent or quasi-consistent Riordan array with two sequences $A$ and $Z$. If $A$ is $\textup{PF}_r$ (resp. $\textup{PF}$),  then $R$ is $\textup{TP}_r$ (resp. $\textup{TP}$). In particular, if the $A$ is log-concave, then $(r_{n,0})_{n \ge 0}$ is log-convex, and $(r_{n,k})_{0 \le k \le n}$ is log-concave for each $n$.
\end{thm}

Chen, Liang and Wang proved Theorem \ref{thm-abst} and \ref{thm-PF} by using recurrence relations and calculating certain minors, after which they asked for combinatorial proofs of these results. In this paper we solve their problem by constructing planar networks with non-negative weights and using the famous Lindstr\"om-Gessel-Viennot lemma, which has already been a useful tool for proving total positivity of matrices. Further, we slightly generalize the first part of Theorem \ref{thm-PF}  , and apply it to derive more totally positive Riordan arrays.

This paper is organized as follows. In Section \ref{sec-pn}, we give a brief introduction to planar networks of matrices, the Lindstr\"om-Gessel-Viennot lemma and several well-known results on planar networks, which may be omitted for readers familiar with this field. In Section \ref{sec-main} we give the combinatorial proofs of Theorem \ref{thm-abst} and \ref{thm-PF} by constructing suitable  networks, and proceed to present generalizations and applications.

\section{Networks and the Lindstr\"om-Gessel-Viennot lemma}\label{sec-pn}

In this section we shall introduce planar networks and the Lindstr\"om-Gessel-Viennot lemma, which serves as an important tool for proving non-negativity of determinants.

Let $D$ be a directed graph (digraph) with vertex set $V(D)$ and arc set $A(D)$. Recall that $D$ is said to be \textit{acyclic} if it contains no directed cycles, \textit{locally finite} if the number of paths from $u$ to $v$ is finite for any $u,v \in V(D)$ and \textit{planar} if it can be embedded into the plane such that any two intersecting arcs intersect only at vertices. Then we define a \textit{network} to be a pair $(D,\mathrm{wt})$ such that $D$ is an acyclic locally finite planar digraph, and $\mathrm{wt}=\mathrm{wt}_D$ (called the \textit{weight function}) is a map from $A(D)$ to a commutative ring with identity. Further, $(D,\mathrm{wt})$ is called a \textit{planar network} if $D$ is planar. 

Let $\mathbf{U}=(U_1,\ldots,U_n)$ and $V=(V_1,\ldots,V_n)$ be two sequences of \textit{sources} and \textit{sinks} of $D$ respectively, that is, there is no arcs pointing to $U_i$ and no arcs pointing out of $V_i$ for each $1 \le i \le n$. 
Let $P(U_i,V_j)$ be the set of directed paths from $U_i$ to $V_j$. For any directed path we define its weight as the product of the weights of all arcs in it, and we define $p(U_i,V_j)$ to be the sum of weights of all paths in $P(U_i,V_j)$. Then $(D,\mathrm{wt})$ is said to be a network of an $n \times n$ matrix $M=(m_{i,j})$ if $m_{i,j} = p(U_i,V_j)$ for $1\le i,j \le n$. 

Moreover, a tuple $(P_1,\ldots,P_n)$ of directed paths is called a \textit{non-intersecting family} if $P_i$ and $P_j$ do not intersect for any $i \neq j$, and we define $\mathcal{N}(\mathbf{U},\mathbf{V})$ to be the set of non-intersecting families $(P_1,\ldots,P_n)$ such that $P_i$ is a directed path from $U_i$ to $V_i$ for any $1 \le i \le n$. The two sequences $\mathbf{U}$ and $\mathbf{V}$ are said to be \textit{compatible} if $\mathcal{N}(\mathbf{U},\sigma(\mathbf{V})) = \mathcal{N}((U_1,\ldots,U_n),(V_{\sigma(1)},\ldots,V_{\sigma(n)})$ is empty for any non-identity permutation $\sigma$ of $[n]=\{1,\ldots,n\}$. As a stronger property, $\mathbf{U}$ and $\mathbf{V}$ are called \textit{fully compatible} if for any $I=\{i_1,\ldots,i_k\}_< \subseteq [n]$ and $J=\{j_1,\ldots,j_k\}_< \subseteq [n]$, the two sequences $U_I=(U_{i_1},\ldots,U_{i_k})$ and $V_J=(V_{j_1},\ldots,V_{j_k})$ are compatible \cite{Bre95}. Actually, this concept is equivalent to the case that $(U_{i_1},U_{i_2})$ and $(V_{j_1},V_{j_2})$ are compatible for any $i_1<i_2$ and $j_1<j_2$, which is also named $D$-compatible \cite{Ste90}. Then the Lindstr\"om-Gessel-Viennot lemma is stated as follows.
\begin{lem}[the Lindstr\"om-Gessel-Viennot lemma]\label{lem-lgv}
Let $(D,\mathrm{wt})$ be a network for an $n \times n$ matrix $M$ with a sequence of sources $\mathbf{U} = (U_1,\ldots,U_n)$ and a sequence of sinks $\mathbf{V} = (V_1,\ldots,V_n)$. 
\begin{itemize}
\item[(i)] For any $I,J \subseteq [n]$ with $|I|=|J|=k$ ($1 \le k \le n$), we have
\[
\det M_{I,J} = \sum_{\sigma \in \mathfrak{S}_k} \mathrm{sign}(\sigma) \sum_{(P_1,\ldots,P_k) \in \mathcal{N}(\mathbf{U}_I,\sigma(\mathbf{V}_J))} \prod_{i=1}^k\mathrm{wt}(P_i).
\]
\item[(ii)] Further, if $D$ is planar and $\mathbf{U},\mathbf{V}$ are fully compatible, then 
\[
\det M_{I,J} = \sum_{(P_1,\ldots,P_k) \in \mathcal{N}(\mathbf{U}_I,\mathbf{V}_J)} \prod_{i=1}^k\mathrm{wt}(P_i)
\]
for any $I,J \subseteq [n]$ with $|I| = |J| = k$. In particular, if the weights of all arcs are non-negative, then $M$ is totally positive.
\end{itemize}
\end{lem}
Originally Lindstr\"om \cite{Lin73} proposed Lemma \ref{lem-lgv} (i) to study the representability of matroids. Later this result was applied and developed to a great extent by Gessel and Viennot \cite{GV85,GV89}, where Lemma \ref{lem-lgv} (ii) (and in particular the special case $I=J=[n]$) was frequently used to show non-negativity of determinants,  see also Brenti \cite{Bre95} and Stembridge \cite{Ste90}. 

\begin{exam}\label{ex-bidiag}
As a simple application, we present a planar network for any bi-diagonal matrix
\[
    B_1 = \begin{pmatrix}
    a_{0,0} & a_{0,1} & 0  & 0  & \quad\\
    0 & a_{1,1} & a_{1,2} & 0  & \quad\\
    0 & 0 & a_{2,2} & a_{2,3} & \quad \\
    \vdots  & \vdots  & \quad & \quad & \ddots \\
    \end{pmatrix}
    \quad \textup{or} \quad
    B_2 = \begin{pmatrix}
    a_{0,0} & 0 & 0  & 0 &\quad\\
    a_{1,0} & a_{1,1} & 0 & 0  & \quad\\
    0 & a_{2,1} & a_{2,2} & 0 & \quad \\
    0 & 0 & a_{3,2} & a_{3,3} & \quad\\
    \vdots  & \vdots  & \quad & \quad & \ddots \\
    \end{pmatrix}.
\]
We only provide a planar network for $B_1$, and the other one is completely analogous. For $i \ge 1$, let $D$ be the (infinite) digraph with
\[
    V(D) = \{U_i \mid i \ge 0\} \cup \{V_i \mid i \ge 0\}
\]
and
\begin{align*}
    A(D) = \{U_i \to V_i \mid i \ge 0\} \cup \{U_{i} \to V_{i+1} \mid i \ge 0\},
\end{align*}
where the coordinates are given by $U_i = (0,-i)$ and $V_i = (1,-i)$. The weight function is defined as
\[
\mathrm{wt}(U_i \to V_i) = a_{i,i}, \quad \mathrm{wt}(U_i \to V_{i+1}) = a_{i,i+1},
\]
see Figure \ref{fig-bidiag}. From this construction one can easily deduce that a bi-diagonal matrix is TP if and only if all its entries are non-negative.
\begin{figure}[htbp]
  \centering
\begin{tikzpicture}[scale=1.5]
\fill (-4.5,1) circle (0.2ex);
\fill (-3.5,1) circle (0.2ex);
\fill (-4.5,0) circle (0.2ex);
\fill (-3.5,0) circle (0.2ex);
\fill (-4.5,-1) circle (0.2ex);
\fill (-3.5,-1) circle (0.2ex);
\draw [thick] [->] (-4.5,1) -- (-4,1);
\draw [thick](-4,1) -- (-3.5,1);
\draw [thick][->] (-4.5,0) -- (-4,0);
\draw [thick](-4,0) -- (-3.5,0);
\draw [thick][->] (-4.5,-1) -- (-4,-1);
\draw [thick](-4,-1) -- (-3.5,-1);
\draw [thick][->](-4.5,1) -- (-4,0.5);
\draw [thick](-4,0.5) -- (-3.5,0);
\draw [thick][->](-4.5,0) -- (-4,-0.5);
\draw [thick](-4,-0.5) -- (-3.5,-1);
\node[above] at (-4,1) { $a_{0,0}$};
\node[right] at (-4,0.5) { $a_{0,1}$};
\node[above] at (-4,0) { $a_{1,1}$};
\node[right] at (-4,-0.5) { $a_{1,2}$};
\node[above] at (-4,-1) { $a_{2,2}$};
\node at (-4.5,-1.5) {$\vdots$};
\node at (-3.5,-1.5) {$\vdots$};
\node[left] at (-4.5,1) {$U_0$};
\node[right] at (-3.5,1) {$V_0$};
\node[left] at (-4.5,0) {$U_1$};
\node[right] at (-3.5,0) {$V_1$};
\node[left] at (-4.5,-1) {$U_2$};
\node[right] at (-3.5,-1) {$V_2$};
\end{tikzpicture}
  \caption{A planar network for $B_1$}\label{fig-bidiag}
\end{figure}
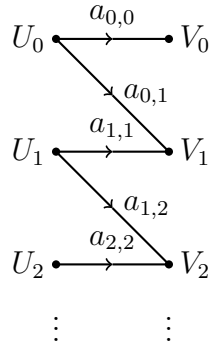
\end{exam}

Further, we have a useful lemma for planar networks of product of two matrices.
\begin{lem}\label{lem-pn-product}
  Let $(D_1,\mathrm{wt}_1)$ (resp. $(D_2,\mathrm{wt}_2)$) be a planar network for an $n \times n$ matrix $M_1=(m^{(1)}_{i,j})$ (resp. $M_2=(m^{(2)}_{i,j})$) with fully compatible source and sink sequences $\mathbf{U}=(U_1,\ldots,U_n)$ and $\mathbf{W}=(W_1,\ldots,W_n)$ (resp. $\mathbf{W'}=(W'_1,\ldots,W'_n)$ and $\mathbf{V}=(V_1,\ldots,V_n)$). Let $D$ be the digraph obtained by placing $D_1,D_2$ together and identifying $W_i$ with $W'_i$ for each $1 \le i \le n$, and let $\mathrm{wt}$ be the (unique) weight function on $A(D)$ whose restriction to $A(D_1)$ (resp. $A(D_2)$) is $\mathrm{wt}_1$ (resp. $\mathrm{wt}_2$). Then $(D,\mathrm{wt},\mathbf{U},\mathbf{V})$ is a planar network for $M_1M_2$. In particular, if the weights of all arcs are non-negative, then $M_1,M_2$ are $\textup{TP}$, and so is $M_1M_2$.
\end{lem}
\begin{proof}
  By the assumption, each directed path from $U_i$ to $V_j$ must pass through exactly one vertex $W_k=W'_k$ for some $1 \le k \le n$. Hence for any $1 \le i,j \le n$ we have
  \[
  p(U_i,V_j) = \sum_{k=1}^n p(U_i,W_k)p(W_k,V_j) = \sum_{k=1}^n m^{(1)}_{i,k}m^{(2)}_{k,j},
  \]
  which corresponds to the matrix multiplication. The last assertion follows from Lemma \ref{lem-lgv} and the fact that $\mathbf{U}$ and $\mathbf{V}$ are also fully compatible by the construction of $D$.
\end{proof}  
At the end of this section, we introduce the following result of Brenti, which may be considered as the converse of the last statement of Lemma \ref{lem-lgv}.
\begin{thm}[\cite{Bre95}]\label{thm-tp-pn}
  Let $M=(m_{i,j})$ be an $n \times n$ totally positive matrix. Then there exists a planar network $(D,\mathrm{wt})$ with fully compatible source and sink sequences $\mathbf{U}=(U_1,\ldots,U_n)$ and $\mathbf{V}=(V_1,\ldots,V_n)$ such that
  \[
  M= (p(U_i,V_j))_{1 \le i,j \le n}
  \]
  and the arcs all have non-negative weights.
\end{thm}

\section{Main results}\label{sec-main}

In this section, we shall first give a combinatorial proof for Theorem \ref{thm-abst}, which involves a planar network construction for tri-diagonal matrices. Later we give a slight generalization of the first part of Theorem \ref{thm-PF}, where the existence of certain planar networks are proved by means of Example \ref{ex-bidiag}, Lemma \ref{lem-pn-product} and Theorem \ref{thm-tp-pn}.
\begin{thm}\label{thm-comb-abst}
  Let $R$ be a Riordan array with production matrix
  \[
    J(R) = 
    \begin{pmatrix}
    a & r & \quad  & \quad  & \quad\\
    b & s & r & \quad  & \quad\\
    \quad  & t & s & r & \\
    \quad  & \quad  & t & s & r \\
    \quad  & \quad  & \quad  & \ddots & \ddots & \ddots
    \end{pmatrix},
  \]
  where $a,b,r,s,t$ are five non-negative numbers.
\begin{itemize}
\item[(i)] If $as \ge br$ and $s^2 \ge rt$, then $J(R)$ has a natural network and is $\textup{TP}_2$. In particular, $(r_{n,0})_{n \ge 0}$ is log-convex.
  \item[(ii)] If $s^2 \ge 4rt$ and $a(s+\sqrt{s^2-4rt})/2 \ge br$, then $J(R)$ has a planar network interpretation with non-negative weights. In particular, the corresponding Riordan array $R(a,b;r,s,t)$ is $\textup{TP}$.
\end{itemize}
\end{thm}

\begin{proof} 
In case (i), we let $D$ be the digraph with
\[
    V(D) = \{U_i \mid i \ge 0\} \cup \{V_i \mid i \ge 0\} 
\]
and
\begin{align*}
    A(D) =& \{U_i \to V_i \mid i \ge 0\} \cup \{U_i \to V_{i+1} \mid i \ge 0\} \cup \{U_{i+1} \to V_i \mid i \ge 0\},
\end{align*}
where the coordinates are given by $U_i = (0,-i)$, $V_i = (1,-i)$ and we require that $U_i \to V_{i+1}$ and $U_{i+1} \to V_i$ do not intersect. The weight function is defined as
\begin{align*}
&\mathrm{wt}(U_0 \to V_0) = a, \quad \mathrm{wt}(U_1 \to V_0) = b, \quad \mathrm{wt}(U_0 \to V_1) = r, \\
&\mathrm{wt}(U_i \to V_{i+1}) = r, \quad \mathrm{wt}(U_i \to V_i) = s, \quad \mathrm{wt}(U_{i+1} \to V_i) = t \text{ for } i \ge 1,
\end{align*}
see Figure \ref{fig-tp2} for an illustration.
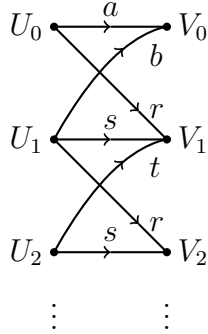
\begin{figure}[htbp]
  \centering
\begin{tikzpicture}[scale=1.5]
\fill (-4.5,1) circle (0.2ex);
\fill (-3.5,1) circle (0.2ex);
\fill (-4.5,0) circle (0.2ex);
\fill (-3.5,0) circle (0.2ex);
\fill (-4.5,-1) circle (0.2ex);
\fill (-3.5,-1) circle (0.2ex);

\draw [thick] [->] (-4.5,1) -- (-4,1);
\draw [thick](-4,1) -- (-3.5,1);
\draw [thick]plot[smooth, tension=1] coordinates {(-4.5,0) (-4,0.7) (-3.5,1)};
\draw [thick] [->] (-3.9,0.8)--(-3.89,0.81);
\draw [thick][->] (-4.5,0) -- (-4,0);
\draw [thick](-4,0) -- (-3.5,0);
\draw [thick][->] (-4.5,-1) -- (-4,-1);
\draw [thick]plot[smooth, tension=1] coordinates {(-4.5,-1) (-4,-0.3) (-3.5,0)};
\draw [thick] [->] (-3.9,0.8-1)--(-3.89,0.81-1);
\draw [thick](-4,-1) -- (-3.5,-1);
\draw [thick][->](-4.5,1) -- (-3.75,0.25);
\draw [thick](-3.75,0.25) -- (-3.5,0);
\draw [thick][->](-4.5,0) -- (-3.75,-0.75);
\draw [thick](-3.75,-0.75) -- (-3.5,-1);

\node[above] at (-4,1) { $a$};
\node[right] at (-3.75,0.75) { $b$};
\node[above] at (-4,0) { $s$};
\node[right] at (-3.75,-0.75) { $r$};
\node[above] at (-4,-1) { $s$};
\node at (-4.5,-1.5) {$\vdots$};
\node at (-3.5,-1.5) {$\vdots$};
\node[left] at (-4.5,1) {$U_0$};
\node[right] at (-3.5,1) {$V_0$};
\node[left] at (-4.5,0) {$U_1$};
\node[right] at (-3.5,0) {$V_1$};
\node[left] at (-4.5,-1) {$U_2$};
\node[right] at (-3.5,-1) {$V_2$};
\node[right] at (-3.75,0.25) {$r$};
\node[right] at (-3.75,-0.25) {$t$};

\end{tikzpicture}
  \caption{An illustration for the network in case (i)}\label{fig-tp2}
\end{figure}

Note that this network is not planar, but we can also apply Lemma \ref{lem-lgv} (i) to prove that $J(R)$ is $\textup{TP}_2$. If we take $i_1 < i_2$ and $j_1 < j_2$, then 
\begin{equation*}
\det J(R)_{(i_1,i_2),(j_1,j_2)} = \sum_{(P_1,P_2) \in \mathcal{N}((i_1,i_2),(j_1,j_2))} \mathrm{wt}(P_1)\mathrm{wt}(P_2) - 
\sum_{(P_1,P_2) \in \mathcal{N}((i_1,i_2),(j_2,j_1))} \mathrm{wt}(P_1)\mathrm{wt}(P_2).
\end{equation*}
By our construction, the first sum on the right hand side is non-negative since $a,b,r,s,t \ge 0$, and the second sum is zero unless $i_1=j_1$ and $i_2=j_2=i_1+1$, in which case
\[
\det J(R)_{(i_1,i_2),(j_1,j_2)} = as-br \textup{ or } s^2-rt \ge 0.
\]
Hence $J(R)$ is $\textup{TP}_2$ and the log-convexity of $(r_{n,0})_{n\ge0}$ follows from Theorem \ref{thm-prod-mat-tp} (ii).

For case (ii), we need to construct a more complicated planar network and apply Lemma \ref{lem-lgv} (ii). Let $D$ be the digraph with
\[
    V(D) = \{U_i \mid i \ge 0\} \cup \{W_i \mid i \ge 0\} \cup \{V_i \mid i \ge 0\}
\]
and
\begin{align*}
    A(D) =& \{U_i \to W_i \mid i \ge 0\} \cup \{W_i \to V_i \mid i \ge 0\} \cup \{U_{i+1} \to W_i \mid i \ge 0\} \cup \{W_i \to V_{i+1} \mid i \ge 0\},
\end{align*}
where the coordinates are given by $U_i = (0,-i)$, $W_i = (1,-i)$ and $V_i = (2,-i)$. Then we proceed to give the weight function.
\begin{itemize}
  \item[(1)] If $r=0$, then $R$ degenerates to a bi-diagonal matrix and by Example \ref{ex-bidiag} naturally admits a planar network interpretation with non-negative weights since $a,b,s,t \ge 0$. Similarly, if $r>0$ and $s=0$, then by $s^2 \ge 4rt$ and $a(s+\sqrt{s^2-4rt})/2 \ge br$ we have $b=t=0$ and in this case $R$ is also bi-diagonal.
  \item[(2)] If $r > 0$ and $s>0$, then we define $\mathrm{wt}(U_i \to W_i) = 1$ and $\mathrm{wt}(W_i \to V_{i+1}) = r > 0$ for all $i \ge 0$. Moreover, we let $\mathrm{wt}(W_i \to V_i) = k_i$ and $\mathrm{wt}(U_{i+1} \to W_i) = l_i$ for $i \ge 0$, where $k_i$ and $l_i$ are defined recursively as follows. See Figure \ref{fig-tp} for an illustration.
      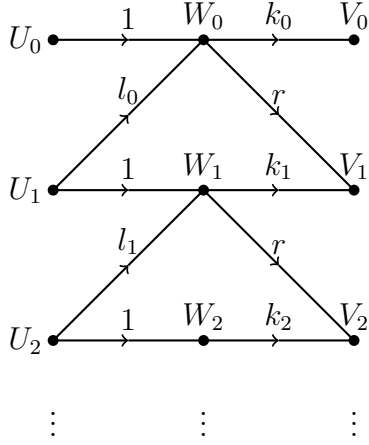
\begin{figure}[htbp]
        \centering
        \begin{tikzpicture}[scale=2]
\fill (-4.5,1) circle (0.2ex);
\fill (-3.5,1) circle (0.2ex);
\fill (-4.5,0) circle (0.2ex);
\fill (-3.5,0) circle (0.2ex);
\fill (-4.5,-1) circle (0.2ex);
\fill (-3.5,-1) circle (0.2ex);
\fill (-2.5,1) circle (0.2ex);
\fill (-2.5,0) circle (0.2ex);
\fill (-2.5,-1) circle (0.2ex);
\draw [thick] [->] (-4.5,1) -- (-4,1);
\draw [thick](-4,1) -- (-3.5,1);
\draw [thick] [->] (-3.5,1) -- (-3,1);
\draw [thick](-3,1) -- (-2.5,1);
\draw [thick][->] (-4.5,0) -- (-4,0);
\draw [thick](-4,0) -- (-3.5,0);
\draw [thick][->] (-3.5,0) -- (-3,0);
\draw [thick](-3,0) -- (-2.5,0);
\draw [thick][->] (-4.5,-1) -- (-4,-1);
\draw [thick](-4,-1) -- (-3.5,-1);
\draw [thick][->] (-3.5,-1) -- (-3,-1);
\draw [thick](-3,-1) -- (-2.5,-1);

\draw [thick][->](-4.5,0) -- (-4,0.5);
\draw [thick](-4,0.5) -- (-3.5,1);
\draw [thick][->](-4.5,-1) -- (-4,-0.5);
\draw [thick](-4,-0.5) -- (-3.5,0);

\draw [thick][->](-3.5,1) -- (-3,0.5);
\draw [thick](-3,0.5) -- (-2.5,0);
\draw [thick][->](-3.5,0) -- (-3,-0.5);
\draw [thick](-3,-0.5) -- (-2.5,-1);

\node[above] at (-4,1) {$1$};
\node[above] at (-4,0.5) { $l_0$};
\node[above] at (-4,0) { $1$};
\node[above] at (-4,-0.5) { $l_1$};
\node[above] at (-4,-1) { $1$};
\node at (-4.5,-1.5) {$\vdots$};
\node at (-3.5,-1.5) {$\vdots$};
\node[left] at (-4.5,1) {$U_0$};
\node[above] at (-3.5,1) {$W_0$};
\node[left] at (-4.5,0) {$U_1$};
\node[above] at (-3.5,0) {$W_1$};
\node[left] at (-4.5,-1) {$U_2$};
\node[above] at (-3.5,-1) {$W_2$};
\node[above] at (-2.5,1) {$V_0$};
\node[above] at (-2.5,0) {$V_1$};
\node[above] at (-2.5,-1) {$V_2$};
\node[above] at (-3,0.5) {$r$};
\node[above] at (-3,-0.5) {$r$};
\node[above] at (-3,1) {$k_0$};
\node[above] at (-3,0) {$k_1$};
\node[above] at (-3,-1) {$k_2$};
\node at (-2.5,-1.5) {$\vdots$};
\end{tikzpicture}
        \caption{An illustration for the planar network in case (ii) (2)}\label{fig-tp}
      \end{figure}
      \begin{itemize}
      \item[(2.1)] If $a > 0$, then we take
        \[
        k_i = \begin{cases}
                a, & \mbox{if } i=0 \\
                s-\frac{br}{a}, & \mbox{if } i=1 \\
                s-\frac{rt}{k_{i-1}}, & \mbox{if } i \ge 2
              \end{cases}
        \quad \textup{and} \quad
        l_i = \begin{cases}
                \frac{b}{a}, & \mbox{if } i=0 \\
                \frac{t}{k_i}, & \mbox{if } i \ge 1
              \end{cases}.
        \]
        One can easily verify that $(D,w)$ is a planar network for $R$ and $(U_0,U_1,\ldots),(V_0,V_1,\ldots)$ are fully compatible. Then we only need to show that $k_i,l_i \ge 0$ for all $i \ge 1$.
        
        If $t=0$, then $k_i=s > 0$ for $i \ge 2$ and $l_i = 0$ for $i \ge 1$, and $k_1=(as-br)/a\ge 0$ because $as \ge a(s+\sqrt{s^2-4rt})/2 \ge br$. If $t > 0$, then
        \[
        k_1 = s-\frac{br}{a} \ge s-\frac{s+\sqrt{s^2-4rt}}{2} = \frac{s-\sqrt{s^2-4rt}}{2} > 0
        \]
        since $r,s > 0$. Assume that $k_i \ge (s-\sqrt{s^2-4rt})/{2} > 0$ for some $i \ge 1$. Then
        {\small \[
        k_{i+1} = s-\frac{rt}{k_{i}} \ge s-\frac{2rt}{s-\sqrt{s^2-4rt}} = s-\frac{2rt(s+\sqrt{s^2-4rt})}{4rt} = \frac{s-\sqrt{s^2-4rt}}{2} > 0,
        \]}
        
        \noindent and hence by induction $k_i > 0$ for all $i \ge 1$. It follows that $l_i > 0$ for all $i \ge 1$.
        \item[(2.2)] If $a=0$, then we have $b=0$ since $as-br \ge 0$. Let
        \[
        k_i = \begin{cases}
                0, & \mbox{if } i=0 \\
                s, & \mbox{if } i=1 \\
                s-\frac{rt}{k_{i-1}}, & \mbox{if } i \ge 2
              \end{cases}
        \quad \textup{and} \quad
        l_i = \begin{cases}
                0, & \mbox{if } i=0 \\
                \frac{t}{k_i}, & \mbox{if } i \ge 1
              \end{cases}.
        \]
        Similar to case (2.1), $(D,w)$ is a planar network for $R$ with fully compatible source and sink sequences $(U_0,U_1,\ldots),(V_0,V_1,\ldots)$. Now we only need to show that $k_i \ge 0$ for $i \ge 2$ and $l_i \ge 0$ for $i \ge 1$.
        
        If $t=0$, then $k_i=s > 0$ for $i \ge 2$ and $l_i = 0$ for $i \ge 1$. If $t > 0$, then $k_1 > k_2 > s/2 > 0$ since $r,s>0$ and $s^2 \ge 4rt > 2rt$. Assume that $k_i > s/2 > 0$ for some $i \ge 2$. Then
        \[
        k_{i+1} = s-\frac{rt}{k_{i}} > s- \frac{rt}{s/2} = \frac{s^2-2rt}{s} \ge \frac{s}{2} >0
        \]
        and hence by induction $k_i > 0$ for all $i \ge 1$. In addition, $l_i > 0$ for all $i \ge 1$.
      \end{itemize}
\end{itemize}
Now we have proved the total positivity of $J(R)$, and then the total positivity of $R$ follows from Theorem \ref{thm-prod-mat-tp} (i).
\end{proof}
\begin{rem}
In fact, the proof of total positivity in Theorem \ref{thm-prod-mat-tp} is based on matrix decompositions, which may also be stated in terms of planar networks in view of Lemma \ref{lem-pn-product} and Theorem \ref{thm-tp-pn}.
\end{rem}
Then we proceed to give the combinatorial proof of the result of  total positivity in Theorem \ref{thm-PF}, and actually we shall prove a slightly stronger result.
\begin{thm}\label{thm-comb-PF}
Let $R=(r_{n,k})_{n,k\ge0}$ be a Riordan array with $A=(a_{n})_{n\ge0}$ and $Z=(ta_{0},ta_1,ta_2,\ldots)$ or $Z=(ta_{0}+a_1,a_2,a_3,\ldots)$ for some $t \ge 0$. If $A$ is $\textup{PF}$, then $R$ is $\textup{TP}$. 
In particular, this result holds for all consistent and quasi-consistent Riordan arrays.
\end{thm}
\begin{proof}
At first, we characterize the production matrices of such Riordan arrays. If the $Z$-sequence is $(ta_{0},ta_1,ta_2,\ldots)$, then the corresponding production matrix $J(R)$ satisfies
\begin{equation}\label{eq-PF-right}
    J(R) =
    \begin{pmatrix}
    ta_{0} & a_0 & \quad & \quad & \quad\\
    ta_{1} & a_1 & a_0 & \quad & \quad\\
    ta_{2} & a_2 & a_1 & a_0 & \quad\\
    \vdots & \vdots & \quad & \quad & \ddots
    \end{pmatrix}
    =
    \begin{pmatrix}
    a_0 & \quad & \quad & \quad\\
    a_1 & a_0 & \quad & \quad\\
    a_2 & a_1 & a_0 & \quad\\
    \vdots & \quad & \quad & \ddots
    \end{pmatrix}
    \begin{pmatrix}
    t & 1 & 0 & 0 & \\
    0 & 0 & 1 & 0 & \\
    0 & 0 & 0 & 1 & \\
    \vdots &  &  &  & \ddots
    \end{pmatrix}.
\end{equation}
Similarly, if $Z=(ta_{0}+a_1,a_2,a_3,\ldots)$, then
\begin{equation}\label{eq-PF-left}
    J(R) =
    \begin{pmatrix}
    ta_{0}+a_1 & a_0 & \quad & \quad & \quad\\
    a_{2} & a_1 & a_0 & \quad & \quad\\
    a_{3} & a_2 & a_1 & a_0 & \quad\\
    \vdots & \vdots & \quad & \quad & \ddots
    \end{pmatrix}
    =
    \begin{pmatrix}
    t & 1 & 0 & 0 & \\
    0 & 0 & 1 & 0 & \\
    0 & 0 & 0 & 1 & \\
    \vdots &  &  &  & \ddots
    \end{pmatrix}
    \begin{pmatrix}
    a_0 & \quad & \quad & \quad\\
    a_1 & a_0 & \quad & \quad\\
    a_2 & a_1 & a_0 & \quad\\
    \vdots & \quad & \quad & \ddots
    \end{pmatrix}.
\end{equation}
By Example \ref{ex-bidiag} and Theorem \ref{thm-tp-pn} the two matrices on the right hand side of both identities admit planar networks with non-negative weights if $t \ge 0$. Moreover, by Lemma \ref{lem-pn-product} these two planar networks may be combined to form a planar network for $J(R)$ in either case. It follows from Lemma \ref{lem-lgv} (ii) that $J(R)$ is $\textup{TP}$ if $A$ is $\textup{PF}$, in which case $R$ is also $\textup{TP}$ by Theorem \ref{thm-prod-mat-tp}.

In particular, if $R$ is a consistent array (resp. quasi-consistent array), we just need to take $t=1$ in \eqref{eq-PF-right} (resp. take $t=0$ in \eqref{eq-PF-left}).
\end{proof}

\begin{exam}
  As a simple application, we provide more totally positive Riordan arrays with the help of OEIS \cite{OEIS}. By the criterion given by Aissen, Schoenberg and Whitney \cite{ASW52}, the sequence $A=(1,1,1,\ldots)$ is \textup{PF} since its generating function is $1/(1-x)$.
  \begin{itemize}
    \item Let $Z=(2,2,2,\ldots)$. Then the corresponding Riordan array $R=(\binom{2n-k}{n})_{n,k \ge 0}$ is totally positive, see the OEIS sequence A092392 for more information.
    \item Let $Z=(2,1,1,\ldots)$. Then the corresponding Riordan array $R$ is totally positive, which gives partial row sums of the reversed Catalan triangle, see the OEIS sequence A054445 for more information.
  \end{itemize}
\end{exam}

\noindent{\bf Acknowledgments.}
We would like to thank Professor Xi Chen for helpful discussions. Ethan Li is supported by the National Natural Science Foundation of China (No. 12501462).

\end{document}